\documentclass[a4paper,reqno,11pt]{amsart}
\usepackage{amssymb}
\usepackage{amstext}
\usepackage{amsmath}
\usepackage{amscd}
\usepackage{amsthm}
\usepackage{amsfonts}
\usepackage{pict2e}%\usepackage{eepic}
\usepackage{enumerate}
\usepackage{graphicx}
\usepackage{latexsym}
\usepackage{mathrsfs}
\input xy
\xyoption{all}
\usepackage{pstricks}
\usepackage{lscape}
\usepackage{comment}

\makeatletter

\@addtoreset{equation}{section}
\makeatother
\newtheorem{theorem}{Theorem}[section]

\newtheorem{corollary}[theorem]{Corollary}
\newtheorem{lemma}[theorem]{Lemma}
\newtheorem{proposition}[theorem]{Proposition}
\newtheorem{definition-theorem}[theorem]{Definition-Theorem}
\newtheorem{definition-proposition}[theorem]{Definition-Proposition}

\newtheorem*{thm}{Theorem}

\theoremstyle{definition}
\newtheorem{definition}[theorem]{Definition}

\newcommand{\A}{{\mathcal A}}
\newcommand{\B}{{\mathcal B}}
\newcommand{\C}{{\mathcal C}}

\newcommand{\ZZ}{\mathbb{Z}}

\newcommand{\NN}{\mathbb{N}}
\newcommand{\inj}{\mathsf{inj}\hspace{.01in}}

\renewcommand{\mod}{\mathsf{mod}\hspace{.01in}}
\newcommand{\Mod}{\mathsf{Mod}\hspace{.01in}}
\newcommand{\proj}{\mathsf{proj}\hspace{.01in}}
\newcommand{\psmod}{\underline{\mod}}
\newcommand{\ismod}{\overline{\mod}}

\newcommand{\Cok}{\operatorname{Cok}\nolimits}

\newcommand{\Ext}{\operatorname{Ext}\nolimits}
\newcommand{\Hom}{\operatorname{Hom}\nolimits}
\renewcommand{\Im}{\operatorname{Im}\nolimits}

\newcommand{\Ker}{\operatorname{Ker}\nolimits}

\newcommand{\op}{\operatorname{op}\nolimits}

\newcommand{\Tr}{\operatorname{Tr}\nolimits}
\newcommand{\identity}{\operatorname{id}\nolimits}
\begin{document}
\title[Recollements for dualizing $k$-varieties]{Recollements for dualizing $k$-varieties and Auslander's formulas}
\author[Yasuaki Ogawa]{Yasuaki Ogawa}
\date{\today}
\address{Graduate School of Mathematics, Nagoya University, Furo-cho, Chikusa-ku, Nagoya, 464-8602, Japan}
\email{m11019b@math.nagoya-u.ac.jp}
\begin{abstract}
Given the pair of a dualizing $k$-variety and its functorially finite subcategory,
we show that there exists a recollement consisting of their functor categories of finitely presented objects.
We provide several applications for Auslander's formulas:
The first one realizes a module category as a Serre quotient of a suitable functor category.
The second one shows a close connection between Auslander-Bridger sequences and recollements.
The third one gives a new proof of the higher defect formula which includes the higher Auslander-Reiten duality as a special case.
\end{abstract}
\maketitle
%%%%%%%%%%%%%%%%%%%%%%%%%%%%%%%%%%%%%%%%%%%%%%%%%%%%%%%%%%%%%%%%%%%%%%%%%
\section*{Introduction}
%%%%%%%%%%%%%%%%%%%%%%%%%%%%%%%%%%%%%%%%%%%%%%%%%%%%%%%%%%%%%%%%%%%%%%%%%
The notion of recollements introduced in \cite{BBD} provides an effective tool for categorical study of algebras.
A recollement of abelian categories is a special case of Serre quotients where both
the inclusion and the quotient functor admit left and right adjoints.
Such a recollement situation is denoted throughout the paper by a diagram of the form below:
$$\xymatrix@C=1.2cm{\B\ar[r]^-{}
&\A\ar[r]^-{}\ar@/^1.2pc/[l]^-{}\ar_-{}@/_1.2pc/[l]
&\C .\ar@/^1.2pc/[l]^{}\ar@/_1.2pc/[l]_{}}$$
It provides a tool for deconstructing the middle category $\A$ into smaller ones $\B$ and $\C$.
There are a lot of recent work on recollements of abelian categories \cite{FP, Psa, PV}.
One of the most studied example is as follows.
Let $A$ be an associative ring (with unit) and $e$ its idempotent. Then there exists a recollement:
\begin{equation}\label{idempotent_recollement}
\xymatrix@C=1.2cm{\Mod A/AeA\ar[r]^-{{}}
&\Mod A\ar[r]^-{}\ar@/^1.2pc/[l]^-{}\ar_-{}@/_1.2pc/[l]
&\Mod eAe .\ar@/^1.2pc/[l]^{}\ar@/_1.2pc/[l]_{}}
\end{equation}
If $A$ is noetherian, it restricts to a recollements consisting of the
categories of finitely generated modules.
In fact, recollements of this type appeared in many contexts in representation theory, e.g. \cite{Kra2, CS, Eir}.

Our aim is to extend this recollement to functor categories over dualizing $k$-varieties.
A dualizing $k$-variety can be considered as an analog of the category of finitely generated projective
modules over a finite dimensional algebra, but with possibly infinitely many indecomposable
objects up to isomorphism \cite{AR}.
It is a Krull-Schmidt $\Hom$-finite $k$-linear category $\A$ where the standard $k$-duality $\Hom_k(-,k)$ induces the duality between $\mod\A$ and $\mod(\A^{\op})$.

\begin{thm}[Theorem \ref{thm:main}]
Let $(\A,\B)$ be the pair of a dualizing $k$-variety $\A$ and its functorially finite subcategory $\B$.
Then we have the following recollement:
$$\xymatrix@C=1.2cm{\mod(\A/[\B])\ar[r]^-{{}}
&\mod\A\ar[r]^-{}\ar@/^1.2pc/[l]^-{}\ar_-{}@/_1.2pc/[l]
&\mod\B .\ar@/^1.2pc/[l]^{}\ar@/_1.2pc/[l]_{}}$$
\end{thm}

In Section 3, we show that in the functor category of a suitable dualizing $k$-variety, Auslander-Bridger sequences are nothing other than right-defining exact sequences of a recollement (Theorem \ref{thm:AB}).

In Section 4, we approach to higher Auslander-Reiten theory from a viewpoint of dualizing $k$-varieties. Auslander-Reiten theory is a fundamental tool for studying representation theory of Artin algebras, see \cite{ARS, ASS}.
In the late 2000s, Higher Auslander-Reiten theory was introduced by Iyama in \cite{Iya2,Iya}.
In this section, we give a higher analog of Auslander's defect formula for an $n$-cluster tilting subcategory $\B$ of $\mod \A$,
where $\A$ is a dualizing $k$-variety
(Theorem \ref{higher defect formula}).
As an application, we give an equivalence
$$\sigma_n :\underline{\B}\xrightarrow{\sim}\overline{\B}$$
and a bifunctorial isomorphism
\begin{equation}\label{0.1}
\underline{\B}(\sigma_n^-y,x)\cong D\Ext^n_\A(x,y)\cong\overline{\B}(y,\sigma_nx).
\end{equation}
In particular, $\sigma_n$ coincides with the $n$-Auslander-Reiten translation $\tau_n$
and (\ref{0.1}) gives an $n$-Auslander-Reiten duality (Theorem \ref{thm:AR-translation}).

A similar approach to higher Auslander-Reiten theory was given by Jasso and Kvamme \cite{JK} independently, but our approach is slightly different since we do not use an explicit form of $\tau_n$.
%%%%%%%%%%%%%%%%%%%%%%%%%%%%%%%%%%%%%%%%%%%%%%%%%%%%%%%%%%%%%%%%%%%%%%%%%
\subsection*{Notation and convention}
%%%%%%%%%%%%%%%%%%%%%%%%%%%%%%%%%%%%%%%%%%%%%%%%%%%%%%%%%%%%%%%%%%%%%%%%%
Throughout this paper we fix a commutative field $k$.
The symbol $A$ always denotes a finite dimensional algebra over the field $k$.
The category of finite dimensional right $A$-modules and its full subcategory of projective  (resp. injective) $A$-modules will be denoted
by $\mod A$ and ${\proj}A$ (resp. ${\inj}A$), respectively.
The projectively (resp. injectively) stable category of $\mod A$ will be denoted by $\psmod A$ (resp. $\ismod A$).

The symbols $\A, \B$ and $\C$ always denote an additive category, and the set of morphisms $a\rightarrow b$ in $\A$ is denoted by $\A(a,b)$.
We consider only additive functors between additive categories;
that is functors $F$ which satisfy $F(f +g) = F(f)+F(g)$ whenever $f +g$ is defined.
For a given category $\A$, we denote its opposite category by $\A^{\op}$.
For a functor $F:\A\rightarrow \C$, its \textit{image} and \textit{kernel} are defined as the full subcategories of $\A$
$$\Im F:=\{y\in\C\mid {^\exists x\in\A},\  Fx\cong y\}\ \ \textnormal{and}\ \  \Ker F:=\{x\in\A\mid F(x)=0\},$$
respectively.
Let $\B$ be a subcategory of $\A$.
We denote by $\A/[\B]$ the ideal quotient category of $\A$ modulo the (two-sided) ideal $[\B]$ in $\A$ consisting of all
morphisms having a factorization through an object in $\B$.
If there exists a fully faithful functor $\B\hookrightarrow\A$,
we often regard $\B$ as a full subcategory of $\A$.

In the case that given categories are $k$-linear; that is for all $x,y\in\A$ its morphism-space $\A(x,y)$ is a $k$-module and the composition
$\A(y,z)\times\A(x,y)\rightarrow\A(x,z)$
is $k$-bilinear,
we consider only additive $k$-linear functors, that is, they give a $k$-linear maps between morphism-spaces. 

%%%%%%%%%%%%%%%%%%%%%%%%%%%%%%%%%%%%%%%%%%%%%%%%%%%%%%%%%%%%%%%%%%%%%%%%%
\subsection*{Acknowledgements}
%%%%%%%%%%%%%%%%%%%%%%%%%%%%%%%%%%%%%%%%%%%%%%%%%%%%%%%%%%%%%%%%%%%%%%%%%
First and foremost, I would like to express my gratitude to my supervisors Kiriko Kato and Osamu Iyama, who gave me so many helpful suggestions and discussions.
I am also grateful to Takahide Adachi for his valuable comments and advice.
%%%%%%%%%%%%%%%%%%%%%%%%%%%%%%%%%%%%%%%%%%%%%%%%%%%%%%%%%%%%%%%%%%%%%%%%%
\section{Preliminaries}
%%%%%%%%%%%%%%%%%%%%%%%%%%%%%%%%%%%%%%%%%%%%%%%%%%%%%%%%%%%%%%%%%%%%%%%%%
In this section we recall the notion of dualizing $k$-varieties introduced by Auslander and Reiten in \cite{AR}.
We also recall the definition of recollements of abelian categories, as well as some basic properties which are needed in this paper.
%%%%%%%%%%%%%%%%%%%%%%%%%%%%%%%%%%%%%%%%%%%%%%%%%%%%%%%%%%%%%%%%%%%%%%%%%
\subsection{Dualizing $k$-varieties}
%%%%%%%%%%%%%%%%%%%%%%%%%%%%%%%%%%%%%%%%%%%%%%%%%%%%%%%%%%%%%%%%%%%%%%%%%
We recall from \cite{AR} some basic facts about dualizing $k$-varieties.
We denote by $\mathsf{Ab}$ the category of abelian groups.
For an essentially small category $\A$,
a \textit{(right) $\A$-module}  is defined to be a contravariant functor $\A \rightarrow \mathsf{Ab}$ and a \textit{morphism} $X\rightarrow Y$ between $\A$-modules $X$
and $Y$ is a natural transformation. Thus we define an abelian category of $\A$-modules which is denoted by $\Mod\A$.
In the category $\Mod\A$, the morphism-space $(\Mod\A)(X,Y)$ is usually denoted by $\Hom_\A(X,Y)$.

In the case that given categories are $k$-linear, it is natural to consider, instead of the category of additive functors to $\mathsf{Ab}$; the equivalent category of $k$-linear functors to $\Mod k$, which is denoted by the same symbol.

An $\A$-module $X$ is \textit{finitely generated} if there exists an epimorphism $\A(-,a)\twoheadrightarrow X$ for some $a\in\A$.
An $\A$-module $X$ is said to be \textit{finitely presented} if there exists an exact sequence
$$\A(-,a_1)\rightarrow\A(-,a_0)\rightarrow X\rightarrow 0$$
for some $a_0,a_1\in\A$.
We denote by $\mod\A$ the full subcategory of finitely presented $\A$-modules.

For an arbitrary category $\A$, although the subcategory $\mod\A$ is closed under cokernels and extensions in $\Mod\A$,
it is not necessarily abelian since it is not necessarily closed under kernels.
Let $f:Y\rightarrow Z$ be a morphism in $\A$.
We call a morphism $g:X\rightarrow Y$ a \textit{weak-kernel for $f$} if the induced sequence
$$\A(-,X)\xrightarrow{g\circ -}\A(-,Y)\xrightarrow{f\circ -}\A(-,Z)$$
is exact.
We say $\A$ \textit{admits weak-kernels} if every morphism in $\A$ has a weak-kernel.
The notion of \textit{weak-cokernel} is defined dually.
We recall the following well-known fact.

\begin{lemma}\label{lem:weak-kernel}
The following are equivalent for a category $\A$.
\begin{itemize}
\item[(i)] The category $\A$ admits weak-kernels.
\item[(ii)] The full subcategory $\mod\A$ is an exact abelian subcategory in $\Mod\A$,
that is, it is abelian and the canonical inclusion $\mod\A\hookrightarrow\Mod\A$ is exact.
\end{itemize}
\end{lemma}

Like the case for module categories, $\proj\A$ (resp. $\inj\A$) denotes the full subcategory of projective (resp. injective) $\A$-modules in $\mod\A$, and the projectively (injectively) stable category will be denoted by $\psmod\A:=(\mod\A )/[\proj\A]$ (resp. $\ismod\A:=(\mod\A )/[\inj\A]$).

In the rest of this subsection,
let $\A$ be a \textit{Krull-Schmidt} $k$-linear category, that is, each object $x\in\A$ admits a decomposition $x\cong \coprod_{i=1}^nx_i$ with $\A(x_i,x_i)$ a local $k$-algebra for any $i\in\{1,\ldots , n\}$.
We also assume that $\A$ is \textit{Hom-finite}, that is, each morphism-space $\A(x,y)$ is a finite dimensional $k$-module.
We denote by $D:=\Hom_k(-,k):\mod k\rightarrow\mod k$ the standard $k$-duality.

\begin{definition}\cite[Section 2]{AR}
A Krull-Schmidt Hom-finite $k$-linear category $\A$ is a \textit{dualizing $k$-variety} if the standard $k$-duality $D:\Mod\A\rightarrow\Mod(\A^{\op}),\ X\mapsto D\circ X$ induces a duality $D:\mod\A\xrightarrow{\sim}\mod(\A^{\op})$.
\end{definition}

It is obvious that $\A$ is a dualizing $k$-variety if and only if so is $\A^{\op}$.
If $\A$ is a dualizing $k$-variety, due to the duality between $\mod\A$ and $\mod (\A^{\op})$, $\mod\A$ is closed under kernels in $\Mod\A$.
Thus $\mod\A$ is an exact abelian subcategory in $\Mod\A$.
By Lemma \ref{lem:weak-kernel}, $\A$ admits weak-kernels and weak-cokernels.
The following proposition gives us basic examples of dualizing $k$-varieties.

\begin{proposition}\label{prop:dv1}\cite[Prop. 2.6]{AR}
Suppose $\A$ is a dualizing $k$-variety. Then $\mod\A$ is a dualizing $k$-variety. Moreover, $\mod\A$ admits injective hulls and projective covers.
\end{proposition}

Next we recall the definition of functorially finite subcategories and collect some lemmas about functorially finite subcategories in dualizing $k$-varieties which will be used later.
The symbol $X|_\B$ denotes the restricted functor of an $\A$-module $X$ onto a subcategory $\B$.
Especially, for a functor category $\mod \A$ and its full subcategory $\B$
we also write $\Ext^i_\A(\B,x):=\Ext^i_\A(-,x)|_\B$, where $x\in\mod \A$ and $i\in\ZZ_{\geq 0}$.

\begin{definition}
Let $\A$ be an arbitrary category and $\B$ a full subcategory in $\A$.
\begin{itemize}
\item[(i)] The full subcategory $\B$ is \textit{contravariantly finite} if the functor $\A(-,x)|_\B$ is a finitely generated $\B$-module for each $x\in\A$.
\item[(ii)] The full subcategory $\B$ is \textit{covariantly finite} if the functor $\A(x,-)|_\B$ is a finitely generated $\B^{\op}$-module for each $x\in\A$.
\item[(iii)] We call $\B$ a \textit{functorially finite} if it is contravariantly finite and covariantly finite.
\end{itemize}
\end{definition}

If a $\B$-module $\A(-,x)|_\B$ is finitely generated, then there exists an epimorphism
$$\B(-,b)\xrightarrow{\alpha\circ -}\A(-,x)|_\B\rightarrow 0$$
in $\Mod\B$ for some $b\in\B$.
Then we call the induced map $\alpha :b\rightarrow x$ a \textit{right $\B$-approximation of $x$}.
Dually we define the notion of \textit{left $\B$-approximation}.

It is easy to verify that the subcategories ${\proj}\A$ and ${\inj}\A$ are functorially finite
in $\mod \A$ if $\A$ is a dualizing $k$-variety.
The following result gives a criterion for a given subcategory to be a dualizing $k$-variety.

\begin{proposition}\label{prop:dv2}\cite[Thm. 2.3]{AS}\cite[Prop. 1.2]{Iya2}
Let $\B$ be a functorially finite subcategory in a dualizing $k$-variety $\A$.
Then $\B$ is a dualizing $k$-variety.
\end{proposition}

%%%%%%%%%%%%%%%%%%%%%%%%%%%%%%%%%%%%%%%%%%%%%%%%%%%%%%%%%%%%%%%%%%%
\subsection{Recollements of abelian categories}
%%%%%%%%%%%%%%%%%%%%%%%%%%%%%%%%%%%%%%%%%%%%%%%%%%%%%%%%%%%%%%%%%%%
In this subsection we recall some basic facts on recollements of abelian categories.
Let us start with introducing basic terminologies.
A pair of functors $L:\A\rightarrow\B$ and $R:\B\rightarrow\A$ is said to be an \textit{adjoint pair} if there exists a bifunctorial isomorphism
$\A(a,Rb)\cong \B(La,b)$
in $a\in\A$ and $b\in\B$.
We simply denote this adjoint pair by $(L\dashv R)$.
For a functor $F:\A\rightarrow \B$, we often denote its right (resp. left) adjoint by $F_\rho$ (resp. $F_\lambda$).
If $F$ admits a right adjoint $F_\rho$ as well as a left adjoint $F_\lambda$,
we denote this situation by
$$\xymatrix@C=1.2cm{\A\ar[r]|-{{F}}
&\B .\ar@/^1.2pc/[l]^-{F_\rho}\ar_-{F_\lambda}@/_1.2pc/[l]}$$

Throughout the rest of this subsection $\A$ is always assumed to be an abelian category.
To begin with, we recall the definition of recollement, following \cite{FP,Psa}
(see also \cite[Ch. 4]{Pop}).
\begin{definition}\label{def:recollement}
Let $\A,\B$ and $\C$ be abelian categories.
A \textit{recollement of $\A$ relative to $\B$ and $\C$} is given by six functors
$$\xymatrix@C=1.2cm{\B\ar[r]|-{{e}}
&\A\ar[r]|-{q}\ar@/^1.2pc/[l]^-{e_\rho}\ar_-{e_\lambda}@/_1.2pc/[l]
&\C \ar@/^1.2pc/[l]^{q_\rho}\ar@/_1.2pc/[l]_{q_\lambda}}$$
such that
\begin{itemize}
\item[(R1)] They form four adjoint pairs $(e_\lambda\dashv e), (e\dashv e_\rho), (q_\lambda\dashv q)$ and $(q\dashv q_\rho)$.
\item[(R2)] The functors $q_\lambda,q_\rho$ and $e$ are fully faithful.
\item[(R3)] $\Im e=\Ker q$.
\end{itemize}
We denote this recollement by $(\B,\A,\C)$ for short.
\end{definition}

Notice that the functors $q$ and $e$ are exact, since each of them admits a right adjoint and a left adjoint.
The following proposition shows that a recollement is a special case of Serre quotients.

\begin{proposition}\label{prop:Serre}\cite[Thm. 4.9]{Pop}
Let $q: \A\rightarrow \C$ be an exact functor.
If it admits a fully faithful right adjoint $q_\rho$ (resp. left adjoint $q_\lambda$), the functor $q$ induces an equivalence between $\C$ and the Serre quotient $\A /\Ker q$ of $\A$ with respect to the Serre subcategory $\Ker q$.
\end{proposition}

The following notions play a central role in Section 3.

\begin{proposition}\cite[Prop. 2.6]{Psa}\label{prop:defining_exact_sequence}
For a given recollement $(\B,\A,\C)$ and an object $x\in\A$,
\begin{itemize}
\item[(i)] we have an exact sequence $0\rightarrow (ee_\rho)x\xrightarrow{\eta}x\xrightarrow{\epsilon}(q_\rho q)x\rightarrow y\rightarrow 0$,
where $\eta$ and $\epsilon$ are the counit and the unit of the adjoint pairs, respectively.
We call it the \textnormal{right-defining exact sequence}.
\item[(ii)] we have an exact sequence $0\rightarrow y^\prime\rightarrow (q_\rho q)x\xrightarrow{\eta^\prime}x\xrightarrow{\epsilon^\prime}(e e_\rho)x\rightarrow 0$,
where $\eta^\prime$ and $\epsilon^\prime$ are the counit and unit of the adjoint pairs, respectively.
We call it the \textnormal{left-defining exact sequence}.
\end{itemize}
Moreover, if there exists an exact sequence $0\rightarrow x^\prime\rightarrow x\rightarrow x^{\prime\prime}\rightarrow x^{\prime\prime\prime}\rightarrow 0$ with
$x^\prime,x^{\prime\prime\prime}\in\Im e$ and $x^{\prime\prime}\in\Im q_\rho$,
then it is isomorphic to the right-defining exact sequence of $x$.
The dual statement holds for the left-defining exact sequences.
\end{proposition}

%%%%%%%%%%%%%%%%%%%%%%%%%%%%%%%%%%%%%%%%%%%%%%%%%%%%%%%%%%%%%%%%%%%%%%%%%
\section{Recollements arising from dualizing $k$-varieties}
%%%%%%%%%%%%%%%%%%%%%%%%%%%%%%%%%%%%%%%%%%%%%%%%%%%%%%%%%%%%%%%%%%%%%%%%%
Let us start with introducing basic terminologies.
We call a biadditive bifunctor from $\B^{\op}\times\A$ to $\mathsf{Ab}$ a \textit{$\B$-$\A$-bimodule}.
We often regard $\A$ as an $\A$-$\A$-bimodule by the following way
$$_\A\A_\A:=\A(-,+):\A^{\op}\times\A\rightarrow \mathsf{Ab},\ (a^{\op}, a^\prime)\mapsto \A(a^\prime,a).$$
Consider a full subcategory $\B$ in $\A$.
The canonical inclusion $i :\B\hookrightarrow\A$ gives a natural $\A$-$\B$-bimodule structure on $\A$ by
$$_\A\A_\B:=\A(i(-),+):\A^{\op}\times\B\rightarrow\mathsf{Ab},\ (a^{\op}, b)\mapsto \A(i(b),a)=\A(b,a).$$
Similarly, we define a $\B$-$\A$-bimodule $_\B\A_\A:=\A(-,i(+))$.
The first step is to show the following elementary proposition,
which is a categorical analog of the recollement (\ref{idempotent_recollement}) and well-known for experts.
Because we could not find a proper reference, we include a detailed proof.

\begin{proposition}\label{prop:recollement2}
Let $(\A,\B)$ be the pair of an additive category $\A$ and its full subcategory $\B$.
Then we have the following recollement:
\begin{equation}\label{recollement2}
\xymatrix@C=1.2cm{\Mod(\A/[\B])\ar[r]|-{{p^*}}
&\Mod\A\ar[r]|-{i^*}\ar@/^1.2pc/[l]^-{p^*_\rho}\ar_-{p^*_\lambda}@/_1.2pc/[l]
&\Mod\B .\ar@/^1.2pc/[l]^{i^*_\rho}\ar@/_1.2pc/[l]_{i^*_\lambda}}
\end{equation}
\end{proposition}
\begin{proof}
(i) We shall construct the adjoint pairs on the right side in (\ref{recollement2}).
The inclusion $i:\B\hookrightarrow \A$ induces the natural restriction functor
$i^*:\Mod\A\rightarrow \Mod\B, X\mapsto X|_\B$ which is isomorphic to $\Hom_\A(_\B\A_\A,-)\cong -\otimes_\A(_\A\A_\B)$.
Thus it admits a left adjoint $i^*_\lambda:=-\otimes_\B(_\B\A_\A)$ and a right adjoint $i^*_\rho:=\Hom_\B(_\A\A_\B,-)$.
By easy calculation, we show that $i^*_\lambda$ and $i^*_\rho$ are fully faithful.
In fact, we have the following isomorphisms:
\begin{eqnarray*}
i^*\circ i^*_\lambda&\cong&-\otimes_\B(_\B\A_\A)\otimes_\A(_\A\A_\B)\cong -\otimes_\B\B\ \cong\ {\identity}_{\Mod\B},\\
i^*\circ i^*_\rho&\cong&\Hom_\A(_\B\A_\A,\Hom_\B(_\A\A_\B,-))\\
&\cong&\Hom_\B((_\B\A_\A)\otimes_\A(_\A\A_\B),-)\\
&\cong&\Hom_\B(\B,-)\ \cong\ {\identity}_{\Mod\B}.
\end{eqnarray*}
Thus we have constructed the right side of (\ref{recollement2}).

(ii) We shall construct the adjoint pairs on the left side in (\ref{recollement2}).
The canonical projection $p:\A\rightarrow\A/[\B]$ also induces the natural restriction functor $p^* :\Mod(\A/[\B])\rightarrow \Mod\A$.
By a similar argument to the above, $p^*$ is a fully faithful exact functor which admits
a left adjoint $p^*_\lambda$ and a right adjoint $p^*_\rho$.
Thus we have obtained the left side of (\ref{recollement2}).

(iii) It remains to show that $\Im p^* =\Ker i^*$.
This follows from the next lemma.

\begin{lemma}\label{lem:stable}
Let $X$ be an object in $\Mod\A$.
Then $X$ belongs to $\Im p^*$ if and only if $X$ vanishes on objects in $\B$.
In particular, we have $\Im p^* =\Ker i^*$.
\end{lemma}
\begin{proof}
If $X\in\Im p^*$, then there exists $X^\prime\in\Mod(\A/[\B])$ such that $X\cong p^*X^\prime =X^\prime\circ p$.
The functor $p$ vanishes on objects in $\B$, so does $X$.
Conversely, if $X(b)=0$ for any $b\in\B$, then we have a unique functor $X^\prime:\A/[\B]\rightarrow\mathsf{Ab}$ such that $X\cong X^\prime p\cong p^*X^\prime$.
\end{proof}

We have thus proved Proposition \ref{prop:recollement2}.
\end{proof}

The aim of this section is to prove our main theorem.

\begin{theorem}\label{thm:main}
Let $(\A,\B)$ be the pair of a dualizing $k$-variety $\A$ and its functorially finite subcategory $\B$.
Then we have the following recollement:
\begin{equation}\label{our_recollement}
\xymatrix@C=1.2cm{\mod(\A/[\B])\ar[r]|-{{e}}
&\mod\A\ar[r]|-{q}\ar@/^1.2pc/[l]^-{e_\rho}\ar_-{e_\lambda}@/_1.2pc/[l]
&\mod\B ,\ar@/^1.2pc/[l]^{q_\rho}\ar@/_1.2pc/[l]_{q_\lambda}}
\end{equation}
where these six functors are restricted ones of the functors which appear in the recollement (\ref{recollement2}).
In particular, we have an equivalence $\displaystyle\frac{\mod\A}{\mod(\A /[\B])}\simeq \mod\B$.
\end{theorem}

We call this  the \textit{recollement arising from the pair $(\A,\B)$ of a dualizing $k$-variety $\A$ and a functorially finite subcategory $\B$ in $\A$}.

In the rest of this section, we give a proof of Theorem \ref{thm:main}. Let $\A$ and $\B$ be Krull-Schmidt Hom-finite $k$-linear categories.

First we consider the right part of the recollement (\ref{recollement2}).
The existence of the isomorphisms $i^*_\lambda(\B(-,b))=\B(-,b)\otimes_\B(_\B\A_\A)\cong \A(-,b)$ shows that the functor $i^*_\lambda$ preserves indecomposable projectives.
Since $i^*_\lambda$ is right-exact, we have the restricted functor $i^*_\lambda:\mod\B\rightarrow\mod\A$, which is denoted by the same symbol.
However, in general $i^*$ and $i^*_\rho$ can not be restricted onto the subcategories of finitely presented functors.
We will show in Proposition \ref{prop:adj3} that if $(\A, \B)$ is a pair of a dualizing $k$-variety $\A$ and its functorially finite subcategory $\B$, $i^*$ and $i^*_\rho$ can be restricted onto the subcategories.

Second we consider the left part of the recollement (\ref{recollement2}).
Like the case for the canonical inclusion $i$, although the left adjoint $p^*_\lambda$ preserves finitely presented functors, $p^*$ and $p^*_\rho$ do not necessarily preserve finitely presentedness.

The next lemma shows a necessary and sufficient condition so that $i^*$ and $p^*$ preserves finitely presentedness, see \cite[Prop. 3.9]{Buc} for the equivalence (i) and (iii) below.

\begin{lemma}\label{lem:adj1}
Let $\A$ be a category with weak-kernels and $\B$ a full subcategory in $\A$.
Then the following are equivalent.
\begin{itemize}
\item[(i)] The category $\B$ is contravariantly finite.
\item[(ii)] We have the restricted functor $i^*:\mod\A\rightarrow\mod\B$.
\item[(iii)] We have the restricted functor $p^*:\mod(\A/[\B])\rightarrow\mod\A$.
\end{itemize}
Moreover, under the above equivalent conditions, there exist the adjoint pairs
$$
\xymatrix@C=1.2cm{\mod(\A/[\B])\ar[r]|-{{p^*}}
&\mod\A\ar_-{p^*_\lambda}@/_1.2pc/[l]&\textnormal{and}&
\mod\A\ar[r]|-{{i^*}}
&\mod\B .\ar_-{i^*_\lambda}@/_1.2pc/[l]
}
$$
\end{lemma}
\begin{proof}
(i) $\Rightarrow$ (ii): Since $i^*$ is exact, we have only to show that $i^*(\A(-,x))$ is finitely presented for any $x\in\A$.
Since $\B$ is contravariantly finite, there exists a right $\B$-approximation $\alpha_0 :b_0\rightarrow x$.
The morphism $\alpha_0$ induces an epimorphism $\B(-,b_0)\xrightarrow{\alpha_0\circ -} \A(-,x)|_\B\rightarrow 0$,
that is, the $\B$-module $\A(-,x)|_\B$ is finitely generated.
Since $\mod\A$ is abelian, we have the kernel-sequence
$0\rightarrow X\rightarrow \A(-,b_0)\xrightarrow{\alpha_0\circ -} \A(-,x)$ in $\mod\A$
induced from the morphism $\alpha_0$.
Since $X\in\mod\A$, there exists an epimorphism $\A(-,x^\prime)\rightarrow X\rightarrow 0$ and thus we have an exact sequence
$$\A(-,x^\prime)|_\B\rightarrow\B(-,b_0)\rightarrow \A(-,x)|_\B\rightarrow 0$$
in $\mod\B$.
The fact that $\A(-,x^\prime)|_\B$ is finitely generated shows that $\A(-,x)|_\B$ is finitely presented.

(ii) $\Rightarrow$ (i): For any $x\in\A$, the functor $i^*(\A(-,x)) = \A(-,x)|_\B$ is finitely presented.
This shows that $\B$ is contravariantly finite by definition.
\end{proof}

If $\A$ is a dualizing $k$-variety and $\B$ is functorially finite
in $\A$,
then the restriction functor $i^*:\mod\A\rightarrow\mod\B$ admits a right adjoint.

\begin{proposition}\label{prop:adj3}
There exist the following adjoint pairs for the pair $(\A,\B)$ of a dualizing $k$-variety $\A$ and its full subcategory $\B$:
$$
\xymatrix@C=1.2cm{\mod\A\ar[r]|-{q}&\mod\B , \ar@/^1.2pc/[l]^{q_\rho}\ar@/_1.2pc/[l]_{q_\lambda}}
$$
where $q:=i^*$ is the restriction functor induced by the canonical inclusion $i:\B\hookrightarrow\A$.
Moreover, we have isomorphisms $q_\rho\cong \Hom_\B(_\A\A_\B,-)$ and $q_\lambda\cong -\otimes_\B(_\B\A_\A)$.
\end{proposition}
\begin{proof}
Recall that every dualizing $k$-variety admits weak-kernels and weak-cokernels.
As we have seen in Lemma \ref{lem:adj1}, there exists an adjoint pair $(q_\lambda\dashv q)$ between $\mod\A$ and $\mod\B$.
Since $\A^{\op}$ is also a dualizing $k$-variety and $\B^{\op}$ is functorially finite in $\A^{\op}$, by Lemma \ref{lem:adj1}, we have an adjoint pair
$(q^\prime_\lambda\dashv q^\prime)$ between $\mod(\A^{\op})$ and $\mod(\B^{\op})$,
where $q^\prime$ is the restriction functor induced by the inclusion $\B^{\op}\hookrightarrow\A^{\op}$.
Since $\A$ and $\B$ are dualizing $k$-varieties, we have the following functors:
$$
\xymatrix@C=1.2cm{
\mod\A\ar[r]|-{q}\ar@{<->}[d]_D
&\mod\B \ar@/_1.2pc/[l]_{q_\lambda}\ar@{<->}[d]^D\\
\mod(\A^{\op})\ar[r]|-{q^\prime}
&\mod(\B^{\op}). \ar@/_1.2pc/[l]_{q^\prime_\lambda}
}
$$
First we notice that $q\cong Dq^\prime D$ holds by definition.
Put $q_\rho:=Dq^\prime_\lambda D:\mod\B\rightarrow\mod\A$. It is easy to check that $q$ and $q_\rho$ form an adjoint pair $(q\dashv q_\rho)$.

In the remaining part of the proof, we shall verify the latter statement, namely, an isomorphism $q_\rho\cong\Hom_\B(_\A\A_\B,-)$.
This can be verified by the following calculations.
Since $q_\rho$ is left-exact and preserves injective objects, we have only to check the values of $q_\rho$ on injective $\B$-modules.
Due to the duality $D:\mod\B\rightarrow\mod (\B^{\op})$, each injective $\B$-module is isomorphic to $D\B(x,-)$ for some $x\in\B$.
\begin{eqnarray*}
q_\rho(D\B(x,-))&=&Dq^\prime_\lambda D(D\B(x,-))\\
&\cong&D((_\A\A_\B)\otimes_\B\B(x,-))\\
&\cong&\Hom_{\B^{\op}}(\B (x,-),D(_\A\A_\B))\\
&\cong&\Hom_\B(_\A\A_\B,D\B(x,-)).
\end{eqnarray*}
Therefore $q_\rho\cong \Hom_\B(_\A\A_\B,-)$ on $\mod\B$, and hence it is fully faithful.
\end{proof}

By the discussion so far, we constructed the right part of the recollement (\ref{our_recollement}).
Next, we shall construct the left part of (\ref{our_recollement}).
We keep the assumption that $\A$ is a dualizing $k$-variety and $\B$ is functorially finite.
Let us begin with a ``finitely presented version'' of Lemma \ref{lem:stable}.

\begin{lemma}\label{lem:stable2}
The following hold.
\begin{itemize}
\item[(i)] A finitely presented $\A$-module $X$ belongs to $\mod(\A/[\B])$ if and only if $X$ vanishes on objects in $\B$. In particular, we have $\mod (\A/[\B])=\Ker q$.
\item[(ii)] The ideal quotient category $\A/[\B]$ is a dualizing $k$-variety.
\end{itemize}
\end{lemma}
\begin{proof}
(i) We only prove the ``if'' part. If $X(b)=0$ for any $b\in\B$, there uniquely exists a functor $X^\prime :\A/[\B]\rightarrow\Mod k$ such that $X\cong p^*X^\prime$.
We have only to show that $X^\prime\in\mod(\A/[\B])$.
Applying $p^*_\lambda$ to the isomorphism $X\cong p^*X^\prime$ yields $p^*_\lambda X\cong p^*_\lambda p^*X^\prime\cong X^\prime$.
Since $p^*_\lambda$ preserves finitely presentedness and $X\in\mod\A$, we conclude $X^\prime\in\mod(\A/[\B])$.

(ii) Let $X\in\mod(\A/[\B])$.
Then $DX$ can be regarded as a finitely presented $\A^{\op}$-module which vanishes on $\B$.
Hence $DX\in\mod (\A/[\B])^{\op}$.
Conversely, we can show that $DX^\prime\in\mod\A/[\B]$ for any $X^\prime\in\mod (\A/[\B])^{\op}$.
\end{proof}

By a similar argument in the proof of Proposition \ref{prop:adj3}, we obtain the following.

\begin{proposition}\label{prop:adj4}
There exist the following adjoint pairs for the pair $(\A,\B)$ of a dualizing $k$-variety $\A$ and its full subcategory $\B$:
$$
\xymatrix@C=1.2cm{\mod(\A/[\B])\ar[r]|-{e}&\mod\A , \ar@/^1.2pc/[l]^{e_\rho}\ar@/_1.2pc/[l]_{e_\lambda}}
$$
where $e:=p^*$ is the restriction functor induced by the canonical projection $p:\A\rightarrow\A/[\B]$.
Moreover, we have isomorphisms $e_\rho\cong \Hom_\A(_{\A/[\B]}(\A/[\B])_\A,-)$
and $e_\lambda\cong -\otimes_\A (_\A(\A/[\B])_{\A/[\B]})$.
\end{proposition}

Now we are ready to prove Theorem \ref{thm:main}.

\begin{proof}[proof of Theorem \ref{thm:main}]
By Proposition \ref{prop:adj3} and Proposition \ref{prop:adj4}, we have four adjoint pairs
$(q_\lambda\dashv q), (q\dashv q_\rho), (e_\lambda\dashv e)$ and $(e\dashv e_\rho)$ with $q_\rho, q_\lambda$ and $e$ fully faithful.
By definition, $\Ker q$ is a full subcategory in $\mod\A$ consisiting of functors
which vanihes on $\B$.
Due to Lemma \ref{lem:stable2}, we have $\Ker q=\mod(\A/[\B])$.
Hence they form a recollement.
\end{proof}

Now we apply Theorem \ref{thm:main} to the following special setting.
Let $A$ be a finite dimentional $k$-algebra and $\B$ a functorially finite subcategory of $\mod A$ containing $A$.
Applying Theorem \ref{thm:main} to the pair $(\B,{\proj}A)$ yields the following recollement:
\begin{equation}\label{AB}
\xymatrix@C=1.2cm{\mod\underline{\B}\ar[r]|-{{e}}
&\mod\B\ar[r]|-{q}\ar@/^1.2pc/[l]^-{e_\rho}\ar_-{e_\lambda}@/_1.2pc/[l]
&\mod A,\ar@/^1.2pc/[l]^{q_\rho}\ar@/_1.2pc/[l]_{q_\lambda}}
\end{equation}
where we identify $\mod (\proj A)$ with $\mod A$ via the equivalence $\mod(\proj A)\xrightarrow{\sim}\mod A, X\mapsto X(A)$.
By Proposition \ref{prop:Serre}, this recollement induces the following.

\begin{corollary}\label{cor:AF}
Under the above assumption, there exists an equivalence
$$\frac{\mod\B}{\mod\underline{\B}}\xrightarrow{\sim}\mod A .$$
We call this \textnormal{generalized Auslander's formula}.
\end{corollary}

By setting $\B =\mod A$, we can recover classical Auslander's formula
$\frac{\mod(\mod A)}{\mod (\psmod A)}\simeq \mod A$
 (see \cite[p. 205]{Aus} and \cite[p. 1]{Len} for definition).
%%%%%%%%%%%%%%%%%%%%%%%%%%%%%%%%%%%%%%%%%%%%%%%%%%%%%%%%%%%%%%%%%%%%%%%%%
\section{Application to the Auslander-Bridger sequences}\label{subsection 3.1}
%%%%%%%%%%%%%%%%%%%%%%%%%%%%%%%%%%%%%%%%%%%%%%%%%%%%%%%%%%%%%%%%%%%%%%%%%
The aim of this section is to show a close relationship between recollements and Auslander-Bridger sequences.
Throughout this section, we fix a dualizing $k$-variety $\A$.
Let $\B$ be a functorially finite subcategory in $\mod\A$ which contains $\proj\A$ and $\inj\A$.

Firstly, we recall the definition of Auslander-Bridger sequence, following \cite[Prop. 2.7]{IJ}.
For the category $\B$,
we denote the $\B$-duality by $(-)^*:=\Hom_\B(-,\B)$.
Note that the $\B$-duality yields a duality $(-)^*:{\proj}\B\xrightarrow{\sim}{\proj}(\B^{\op}),\ \B(-,b)\mapsto \B(b,-)$.
Let $X\in\mod\B$ with a minimal projective presentation $\B(-,b_1)\xrightarrow{\alpha}\B(-,b_0)\rightarrow X\rightarrow 0$ and set $${\Tr}X:=\Cok\alpha^*$$
in $\mod(\B^{\op})$, see \cite{AB}.
We denote by $\epsilon :X\rightarrow X^{**}$ the evaluation map.

\begin{definition-proposition}
For each object $X\in\mod\B$, there exists an exact sequence
$$0\rightarrow \Ext^1_{\B^{\op}}({\Tr}X,\B^{\op})\rightarrow X\xrightarrow{\epsilon}X^{**}\rightarrow \Ext^2_{\B^{\op}}({\Tr}X,\B^{\op})\rightarrow 0,$$
which is called the \textnormal{Auslander-Bridger sequence of $X$}.
\end{definition-proposition}

For the convenience of the reader, we recall from \cite[Prop. 6.3]{Aus} and \cite[Prop. 2.7]{IJ} the construction of the Auslander-Bridger sequence of $X$.
Take a minimal projective presentation $\B(-,b_0)\rightarrow\B(-,b_1)\rightarrow X\rightarrow 0$ of $X$.
Taking a left approximation of a cokernel of $b_0\xrightarrow{\alpha} b_1$
yields an exact sequence $b_0\rightarrow b_1\rightarrow b_2$.
Taking a left approximation of a cokernel of $b_1\xrightarrow{\alpha} b_2$
yields an exact sequence
\begin{equation}\label{ABseq1}
b_0\rightarrow b_1\rightarrow b_2\rightarrow b_3
\end{equation}
in $\mod \A$.
By the construction, the sequence (\ref{ABseq1}) induces the exact sequence
\begin{equation}\label{ABseq2}
\B(b_3,-)\xrightarrow{h}\B(b_2,-)\xrightarrow{g}\B(b_1,-)\xrightarrow{f}\B(b_0,-)
\end{equation}
in $\mod (\B^{\op})$.
Note that $X^*=\Ker f$ and ${\Tr}X=\Cok f$.
Taking the $\B^{\op}$-duality of (\ref{ABseq2}) yields the following sequence
\begin{equation}\label{ABseq3}
0\rightarrow \Hom_{\B^{\op}}({\Tr}X,\B^{\op})\rightarrow \B(-,b_0)\xrightarrow{f^*}\B(-,b_1)\xrightarrow{g^*}\B(-,b_2)\xrightarrow{h^*}\B(-,b_3)
\end{equation}
Note that $\Cok f^*\cong X$ and $\Ker h^*=X^{**}$.
Since (\ref{ABseq3}) is a complex, we have a canonical inculusion $i :\Im g^*\hookrightarrow\Ker h^*$ and a unique canonical epimorphism $\epsilon^\prime: X\twoheadrightarrow \Im g^*$.
It is readily verified that there exists a commutative diagram with exact rows.
$$
\xymatrix{
0\ar[r]&\Im f^*\ar@{^{(}->}[d]\ar[r]&\B(-,b_1)\ar@{=}[d]\ar[r]&X\ar@{->>}[d]^{\epsilon^\prime}\ar[r]&0\\
0\ar[r]&\Ker g^*\ar[r]&\B(-,b_1)\ar[r]&\Im g^*\ar[r]\ar@{^{(}->}[d]^i&0\\
&&&X^{**}&
}$$
The evaluation map $\epsilon$ is obtained as the composition $i\circ\epsilon^\prime$.
By the Snake Lemma, we have $\Ker\epsilon\cong\Ker\epsilon^\prime\cong\Ker g^*/\Im f^*$.
It is easy to verify that $\Cok\epsilon\cong\Cok i =\Ker h^*/\Im g^*$.
Since (\ref{ABseq3}) is the $\B^{\op}$-duality of the projective resolution of ${\Tr}X$, we have the isomorphisms $\Ker g^*/\Im f^*\cong \Ext^1_{\B^{\op}}({\Tr}X,\B^{\op})$ and $\Ker h^*/\Im g^*\cong \Ext^2_{\B^{\op}}({\Tr}X,\B^{\op})$.
We have thus obtained the Auslander-Bridger sequence
$$0\rightarrow \Ext^1_{\B^{\op}}({\Tr}X,\B^{\op})\rightarrow X\xrightarrow{\epsilon}X^{**}\rightarrow \Ext^2_{\B^{\op}}({\Tr}X,\B^{\op})\rightarrow 0.$$

We give an interpretation of the Auslander-Bridger sequences via the recollement appearing below.
Due to Theorem \ref{thm:main}, the pair $(\B, \proj\A)$ induces the following recollement:
\begin{equation}\label{AB2}
\xymatrix@C=1.2cm{\mod\underline{\B}\ar[r]|-{{e}}
&\mod\B\ar[r]|-{q}\ar@/^1.2pc/[l]^-{e_\rho}\ar_-{e_\lambda}@/_1.2pc/[l]
&\mod\A,\ar@/^1.2pc/[l]^{q_\rho}\ar@/_1.2pc/[l]_{q_\lambda}}
\end{equation}
where we identify $\mod (\proj\A)$ with $\mod\A$ via the equivalence $\mod(\proj\A)\xrightarrow{\sim}\mod\A$.

\begin{theorem}\label{thm:AB}
Let $(\mod\underline{\B},\mod\B,\mod\A)$ be a recollement (\ref{AB2}).
Then the right-defining exact sequence
$$0\rightarrow (e e_\rho)X\rightarrow X\rightarrow (q_\rho q)X\rightarrow X^\prime\rightarrow 0$$
of $X\in\mod\B$ is isomorphic to the Auslander-Bridger sequence of $X$.
\end{theorem}
In the rest, we give a proof of Theorem \ref{thm:AB}.
By Lemma \ref{lem:stable2} (i), $\mod\underline{\B}$ is a full subcategory in $\mod\B$
consisting of objects $X$ which admits a projective presentation
$$\B(-,b_1)\rightarrow\B(-,b_0)\rightarrow X\rightarrow 0$$
with $b_1\rightarrow b_0$ an epimorphism in $\B$.
Proposition \ref{prop:adj3} gives an explicit description of the functor $q_\rho$.

\begin{lemma}\label{lem:AB-seq1}
The functor $q_\rho :\mod\A\rightarrow \mod\B$ sends $x$ to $\Hom_{\A}(\B,x)$.
\end{lemma}

We define the \textit{2nd syzygy category $\Omega^2(\mod\B)$ of $\mod\B$} to be
the full subcategory of $\mod\B$ consisting of objects $X$ which admits an exact sequence $0\rightarrow X\rightarrow \B(-,b_0)\rightarrow \B(-,b_1)$ for some $b_0,b_1\in\B$.

\begin{lemma}\label{lem:AB-seq2}
We have the equality $\Im q_\rho=\Omega^2(\mod\B)$.
\end{lemma}
\begin{proof}
We show that $\Im q_\rho\subseteq \Omega^2(\mod\B)$. Let $x\in\mod\A$.
Due to $\inj\A\subseteq \B$, there exists an exact sequence $0\rightarrow x\rightarrow b_0\rightarrow b_1$ in $\mod\A$ with $b_0,b_1\in\B$
which is obtained by taking injective copresentation of $x$.
Applying $q_\rho$ to the above exact sequence gives an exact sequence $0\rightarrow q_\rho x\rightarrow q_\rho b_0\rightarrow q_\rho b_1$.
By Lemma \ref{lem:AB-seq1}, $q_\rho(b_i)\cong \B(-,b_i)$ for $i=0,1$.
This implies that $q_\rho x\in\Omega^2(\mod\B)$.

To show the converse, take an object $X\in\Omega^2 (\mod\B)$ with an exact sequence
$$0\rightarrow X\rightarrow \B(-,b_1)\rightarrow \B(-,b_0).$$
Taking the kernel of $b_1\rightarrow b_0$ yields an exact sequence
$0\rightarrow x\rightarrow b_1\rightarrow b_0$
in $\mod\A$. By applying $q_\rho$ to this, we have an exact sequence
$$0\rightarrow q_\rho x\rightarrow\B(-,b_1)\rightarrow \B(-,b_0)$$
in $\mod\B$. Thus we have $X\cong q_\rho x\in\Im q_\rho$.
This finishes the proof.
\end{proof}

Now we are ready to prove Theorem \ref{thm:AB}.

\begin{proof}
Due to Proposition \ref{prop:defining_exact_sequence}, it is enough to show that $X^{**}\in \Im q_\rho$ and $q(\epsilon)$ is an isomorphism.

(i) The fact that $X^{**}\in\Omega^2 (\mod\B)$ follows from the exact sequence.
Actually, we can find the exact sequence $0\rightarrow X^{**}\rightarrow\B(-,b_2)\xrightarrow{h^*}\B(-,b_3)$ in (\ref{ABseq3}).

(ii) Since $q$ is a restriction functor with respect to the subcategory ${\proj}\A$,
we evaluate the sequence (\ref{ABseq3}) on $p\in{\proj}\A$.
This yields an exact sequence
$$\B(p,b_0)\rightarrow \B(p,b_1)\rightarrow \B(p,b_2)\rightarrow \B(p,b_3),$$
since the sequence $b_0\rightarrow b_1\rightarrow b_2\rightarrow b_3$ is exact.
Therefore $\epsilon (p)$ is an isomorphism. Equivalently $q(\epsilon)$ is an isomorphism.
\end{proof}

%%%%%%%%%%%%%%%%%%%%%%%%%%%%%%%%%%%%%%%%%%%%%%%%%%%%%%%%%%%%%%%%%%%%%%%%%
\section{Application to the $n$-Auslander-Reiten duality}
%%%%%%%%%%%%%%%%%%%%%%%%%%%%%%%%%%%%%%%%%%%%%%%%%%%%%%%%%%%%%%%%%%%%%%%%%
Throughout this section let $\A$ be a dualizing $k$-variety and $n$ a positive integer.
We recall the notion of $n$-cluster tilting subcategory in $\mod \A$.
Let $\B$ be a subcategory of $\mod \A$.
For convenience, we define the full subcategories $\B^{\perp_n}$ and $^{\perp_n}\B$ by
\begin{eqnarray*}
\B^{\perp_n}&:=& \{x\in\mod \A\mid i\in\{1,\ldots ,n\} \Ext^i_A(\B,x)=0\},\\
^{\perp_n}\B&:=& \{x\in\mod \A\mid i\in\{1,\ldots ,n\} \Ext^i_A(x,\B)=0\}.
\end{eqnarray*}

\begin{definition}\cite[Def. 2.2]{Iya}
A functorially finite subcategory $\B$ in $\mod \A$ together with $n\in\NN$
is said to be \textit{$n$-cluster-tilting} if the equalities $\B={^{\perp_{n-1}}\B}=\B^{\perp_{n-1}}$ hold.
\end{definition}

Note that $1$-cluster tilting subcategory is nothing other than $\mod \A$.
It is obvious that every $n$-cluster tilting subcategory contains $\proj\A$ and $\inj\A$,
since $\Ext^i_{\A}(\proj\A,-)$ and $\Ext^i_{\A}(-,\inj\A)$ is zero for any $i>0$.
This fact forces each right $\B$-approximation $b\rightarrow x$ of $x$ to be an epimorphism in $\mod \A$, for every $x\in\mod \A$.
Dually each left $\B$-approximation is a monomorphism.

Throughout this section, $\B$ always denotes an $n$-cluster-tilting subcategory in $\mod \A$. We collect some facts for later use.
The following notion is instrumental in this section.

\begin{definition}\cite[Def. 2.4]{Jas}
Let $\B$ be an $n$-cluster-tilting subcategory in $\mod \A$.
A complex $\delta :0\rightarrow b_{n+1}\rightarrow b_n\rightarrow \cdots\rightarrow b_0\rightarrow 0$ in $\B$ is said to be \textit{$n$-exact} if the induced sequences
\begin{eqnarray*}
&&0\rightarrow \B(-,b_{n+1})\rightarrow \B(-,b_n)\rightarrow \cdots\rightarrow \B(-,b_0),\\
&&0\rightarrow \B(b_0,-)\rightarrow\cdots\rightarrow \B(b_n,-)\rightarrow \B(b_{n+1},-)
\end{eqnarray*}
are exact in $\mod\B$ and $\mod(\B^{\op})$, respectively.
\end{definition}

\begin{lemma}\label{lem:cltilt1}\cite[Prop. 3.17]{Jas}
The following hold for an $n$-cluster-tilting subcategory $\B$.
\begin{itemize}
\item[(i)] Each monomorphism $b_{n+1}\rightarrow b_n$ in $\B$ can be embedded in an $n$-exact sequence $\delta :0\rightarrow b_{n+1}\rightarrow b_n\rightarrow \cdots\rightarrow b_0\rightarrow 0$. Moreover, $\delta$ is uniquely determined up to homotopy.
\item[(ii)] Each epimorphism $b_1\rightarrow b_0$ in $\B$ can be embedded in an $n$-exact sequence $\delta :0\rightarrow b_{n+1}\rightarrow \cdots\rightarrow b_1\rightarrow b_0\rightarrow 0$. Moreover, $\delta$ is uniquely determined up to homotopy.
\end{itemize}
\end{lemma}

As a generalization of Auslander's defect introduced in \cite{Aus} (see also Section IV. 4 in \cite{ARS}), we define the following concepts, which were introduced by Jasso and Kvamme independently.

\begin{definition}\cite{JK}
Let $\delta :0\rightarrow b_{n+1}\rightarrow b_n\rightarrow \cdots\rightarrow b_0\rightarrow 0$ be an $n$-exact sequence in $\B$.
The \textit{contravariant $n$-defect} $\delta^{*n}$ and the \textit{covariant $n$-defect} $\delta_{*n}$ are defined by the exactness of the following sequences:
\begin{eqnarray*}
&&0\rightarrow \B(-,b_{n+1})\rightarrow \B(-,b_n)\rightarrow \cdots\rightarrow \B(-,b_0)\rightarrow \delta^{*n}\rightarrow 0,\\
&&0\rightarrow \B(b_0,-)\rightarrow \cdots\rightarrow \B(b_n,-)\rightarrow \B(b_{n+1},-)\rightarrow \delta_{*n}\rightarrow 0.
\end{eqnarray*}
\end{definition}

We give the following characterization of $n$-defects.

\begin{proposition}\label{prop:defect}
The full subcategory of contravariant $n$-defects equals to $\mod\underline{\B}$ in $\mod\B$. Dually the full subcategory of covariant $n$-defects equals to $\mod(\overline{\B}^{\op})$ in $\mod(\B^{\op})$.
\end{proposition}
\begin{proof}
We only prove the former statement.
Consider $X\in\mod\B$ with a projective presentation
$\B(-,b_1)\xrightarrow{\B(-,f_0)}\B(-,b_0)\rightarrow X\rightarrow 0.$
Assume that $X$ belongs to $\mod\underline{\B}$.
Since $X$ vanishes on $p\in\proj\A$, the map $f_0$ is an epimorphism in $\mod \A$.
By Lemma \ref{lem:cltilt1}, the map $f_0$ can be embedded in an $n$-exact sequence $\delta:0\rightarrow b_{n+1}\rightarrow \cdots\rightarrow b_1\rightarrow b_0\rightarrow 0$.
Hence $\delta^{*n}\cong X$.

Conversely, the fact that contravariant $n$-defect $\delta^{*n}$ belongs to $\mod\underline{\B}$ can be confirmed as follows.
Let $\delta:0\rightarrow b_{n+1}\rightarrow\cdots\rightarrow b_1\rightarrow b_0\rightarrow 0$ be the corresponding $n$-exact sequence and $p$ an object in $\proj\A$.
Then we have the right exact sequence
$$\Hom_\A(p,b_1)\rightarrow \Hom_\A(p,b_0)\rightarrow \delta^{*n}(p)\rightarrow 0$$
by definition.
Since $b_1\rightarrow b_0$ is epic in $\mod\A$,
this concludes that $\delta^{*n}(p)=0$. Hence $\delta^{*n}\in\mod\underline{\B}$.
\end{proof}

There exists a duality between $\mod\underline{\B}$ and $\mod(\overline{\B}^{\op})$.
We denote $\mathsf{C}(\B)$ the category of complexes in $\B$.
For convenience, we consider the homotopy category $\mathsf{K}(\B)$ of $\mathsf{C}(\B)$ and
its full subcategory $\mathsf{K}^{n\textnormal{-ex}}(\B)$ consisting of $n$-exact sequences
$\delta :0\rightarrow b_{n+1}\rightarrow\cdots\rightarrow b_1\rightarrow b_0\rightarrow 0$
with the degree of $b_0$ being zero.

\begin{proposition}\label{prop:defect3}
For $n$-exact sequences $\delta :0\rightarrow b_{n+1}\rightarrow\cdots\rightarrow b_1\rightarrow b_0\rightarrow 0$ and $\delta^\prime :0\rightarrow b_{n+1}^\prime\rightarrow\cdots\rightarrow b_1^\prime\rightarrow b_0^\prime\rightarrow 0$,
the following are equivalent.
\begin{itemize}
\item[(i)] The sequence $\delta$ is homotopy equivalent to $\delta^\prime$.
\item[(ii)] There exists an isomorphism $\delta^{*n}\cong \delta^{\prime *n}$.
\item[(iii)] There exists an isomorphism $\delta_{*n}\cong \delta^{\prime}_{*n}$.
\end{itemize}
Moreover, we have a duality $\Phi :\mod(\overline{\B}^{\op})\simeq\mod\underline{\B}$ sending $\delta_{*n}$ to $\delta^{*n}$.
\end{proposition}
\begin{proof}
We assume that $\delta$ is homotopy equivalent to $\delta^\prime$,
that is, there exists chain maps $\phi :\delta\rightarrow\delta^\prime$ and $\psi :\delta^\prime\rightarrow\delta$:
$$\xymatrix{
0\ar[r]&b_{n+1}\ar[r]\ar[d]_{\phi_{n+1}}&\cdots\ar[r]&b_1\ar[r]\ar[d]^{\phi_1}&b_0\ar[r]\ar[d]^{\phi_0}&0\\
0\ar[r]&b_{n+1}^\prime\ar[r]\ar[d]_{\psi_{n+1}}&\cdots\ar[r]&b_1^\prime\ar[r]\ar[d]^{\psi_1}&b_0^\prime\ar[r]\ar[d]^{\psi_0}&0\\
0\ar[r]&b_{n+1}\ar[r]&\cdots\ar[r]&b_1\ar[r]&b_0\ar[r]&0
}$$
with $1-\psi\phi$ and $1-\phi\psi$ null-homotopic.
By a standard argument, we have an isomorphism $\delta^{*n}\cong \delta^{\prime *n}$, an equivalence $\mod\underline{\B}\xrightarrow{\sim}\mathsf{K}^{n\textnormal{-ex}}(\B)$
which sends $\delta^{*n}$ to $\delta$, and
a duality $\mod(\overline{\B}^{\op})\xrightarrow{\sim}\mathsf{K}^{n\textnormal{-ex}}(\B)$
which sends $\delta_{*n}$ to $\delta$.
It is obvious that the composed functor
$$\Phi :\mod(\overline{\B}^{\op})\xrightarrow{\sim}\mathsf{K}^{n\textnormal{-ex}}(\B)\xrightarrow{\sim}\mod\underline{\B},
\ \ \ \delta_{*n}\mapsto \delta\mapsto \delta^{*n}$$
 is a duality.
\end{proof}

In the rest we shall construct the $n$-Auslander-Reiten duality from a viewpoint of dualizing $k$-variety.
As we have seen in Lemma \ref{lem:stable2}, the category $\underline{\B}$ is a dualizing $k$-variety and thus we have the duality $D:\mod\underline{\B}\xrightarrow{\sim}\mod (\underline{\B}^{\op})$.
By composing the duality $\Phi$ in Propositon \ref{prop:defect3} with the duality $D$,
we have the following equivalence.

\begin{proposition}\label{prop:AR-translation}
The composed functor $D\circ\Phi:\mod (\overline{\B}^{\op})\rightarrow \mod (\underline{\B}^{\op})$ is an equivalence which indudes the equivalence $\sigma_n:\underline{\B}\xrightarrow{\sim}\overline{\B}$.
Moreover, we have the following commutative diagram up to isomorphisms:
$$\xymatrix{
\mod\underline{\B}\ar[d]_D&\mod(\overline{\B}^{\op})\ar[l]_\Phi\ar[dl]^{-\circ \sigma_n}\\
\mod (\underline{\B}^{\op})&
}$$
\end{proposition}
\begin{proof}
It is clear that $D\circ\Phi$ gives the equivalence from $\mod (\overline{\B}^{\op})$ to $\mod (\underline{\B}^{\op})$.
We restrict this onto their projective objects, that is, ${\proj}(\overline{\B}^{\op})\simeq {\proj}(\underline{\B}^{\op})$.
Thus we have the equivalence $\sigma_n :\underline{\B}\xrightarrow{\sim}\overline{\B}$
which makes the above diagram commutative up to isomorphisms.
\end{proof}

By the dual argument, we have the equivalence $\sigma_n^-:\overline{\B}\rightarrow\underline{\B}$ which makes the following diagram commutative
up to isomorphisms:
$$\xymatrix{
\mod\underline{\B}\ar[rd]_{-\circ\sigma_n^-}&\mod(\overline{\B}^{\op})\ar@{<-}[l]_{\Phi^{-1}}\ar[d]^D\\
&\mod \overline{\B}
}$$

As an immediate consequence of the above diagrams, we have the higher defect formula.
Moreover, as a special case of the higher defect formula we obtain the higher Auslander-Reiten duality by using a modification of Krause's proof of the classical formula (see \cite{Kra}).

\begin{theorem}\label{higher defect formula}There exist the following formulas:
\begin{itemize}
\item[(i)] \textnormal{(Higher defect formula)} functorial isomorphisms $D\delta^{*n}\cong \delta_{*n}\circ\sigma_n$ and $D\delta_{*n}\cong \delta^{*n}\circ\sigma_n^-$.
\item[(ii)] \textnormal{(Higher Auslander-Reiten duality)} bifunctorial isomorphisms
$\underline{\B}(\sigma^-_n y,x)\cong D\Ext^n_\A(x,y)\cong \overline{\B}(y,\sigma_n x)$
in $x, y\in\B$.
\end{itemize}
\end{theorem}
\begin{proof}
(i) It directly follows from the fact that the duality $\Phi :\mod (\overline{\B}^{\op})\rightarrow \mod \underline{\B}$
sends $\delta_{*n}$ to $\delta^{*n}$ (Proposition \ref{prop:defect3}).

(ii) We only prove the second isomorphism.
Fix an object $y\in\B$.
Let $y\hookrightarrow I(y)$ be an injective hull of $y$ in $\mod\A$.
Complete the $n$-exact sequence $\delta :0\rightarrow y\hookrightarrow I(y)\rightarrow b_{n-1}\rightarrow\cdots\rightarrow b_0\rightarrow 0.$
By Proposition \ref{prop:AR-translation}, we have the isomorphisms $D\delta^{*n}\cong D\Phi (\delta_{*n})\cong \delta_{*n}\circ\sigma_n$.
By \cite[Lem. 3.5]{Iya3}, we have the exact sequence
$$0\rightarrow \B(-,y)\rightarrow \B(-,I(y))\rightarrow \B(-,b_{n-1})\rightarrow \cdots\rightarrow \B(-,b_0)\rightarrow \Ext_\A^n(-,y)\rightarrow \Ext_\A^n(-,I(y))$$
on $\B$.
Since $\Ext^n_\A(-,I(y))=0$, we conclude $\delta^{*n}\cong\Ext^n_\A(-,y)$.
Since $y\hookrightarrow I(y)$ is an injective hull, the exact sequence
$$0\rightarrow\B(b_0,-)\rightarrow \cdots\rightarrow \B(b_{n-1},-)\rightarrow \B(I(y),-)\rightarrow \B(y,-)\rightarrow \delta_{*n}\rightarrow 0$$
shows the isomorphism $\delta_{*n}\cong \overline{\B}(y,-)$.
Therefore we obtain the desired isomorphism $D\Ext^n_A(-,y)\cong \overline{\B}(y,\sigma_n (-))$.
\end{proof}

The isomorphisms in Theorem \ref{higher defect formula}(ii) are nothing other than $n$-Auslander-Reiten dualty. In partilular, the functor $\sigma_n$ (resp. $\sigma^n$) coincides with the $n$-Auslander-Reiten translation $\tau_n$ (resp. $\tau^n$).

We recall the notion of the $n$-Auslander-Reiten duality. Let
$$\tau:{\psmod}\A\rightarrow{\ismod}\A\ \ {\rm and}\ \ 
\tau^-:{\ismod}\A\rightarrow{\psmod}\A$$
be the Auslander-Reiten translations.
As a higher version of the Auslander-Reiten translation, the notion of $n$-Auslander-Reiten translation is defined as follows.
We denote the $n$-th syzygy (resp. $n$-th cosyzygy) functor by $\Omega^n:{\psmod}\A\rightarrow{\psmod}\A$ (resp. $\Omega^{-n}:{\ismod}\A\rightarrow{\ismod}\A$).

\begin{definition-theorem}\cite[Thm. 1.4.1]{Iya}
The \textnormal{$n$-Auslander-Reiten translations} are defined to be the functors
\begin{eqnarray*}
\tau_n&:=&\tau\Omega^{n-1}:{\psmod}\A\xrightarrow{\Omega^{n-1}}{\psmod}\A\xrightarrow{\tau}{\ismod}\A,\\
\tau_n^-&:=&\tau^-\Omega^{-(n-1)}:{\ismod}\A\xrightarrow{\Omega^{-(n-1)}}{\ismod}\A\xrightarrow{\tau^-}{\psmod}\A.
\end{eqnarray*}
These functors induce mutually quasi-inverse equivalences
$$\tau_n :\underline{\B}\rightarrow\overline{\B}\ \ \textnormal{and}\ \ 
\tau_n^- :\overline{\B}\rightarrow\underline{\B}.$$
\end{definition-theorem}

We have the following analog of Theorem \ref{higher defect formula}(ii).

\begin{proposition}\cite[Thm. 1.5]{Iya}\label{prop:AR}
There exist bifunctorial isomorphisms
$\underline{\B}(\tau_n^-y,x)\cong D\Ext^n_A(x,y)\cong\overline{\B}(y,\tau_nx)$
in $x,y\in\B$.
\end{proposition}

Combining results above, we obtain the following explicit form of $\sigma_n$ and $\sigma_n^-$.

\begin{theorem}\label{thm:AR-translation}
The functor $\sigma_n$ and $\sigma_n^-$ are isomorphic to the $n$-Auslander-Reiten translations $\tau_n$ and $\tau_n^-$, repectively.
\end{theorem}
\begin{proof}
Theorem \ref{higher defect formula} and Theorem \ref{prop:AR} gives an isomorphism
$\overline{\B}(y,\sigma_nx)\cong\overline{\B}(y,\tau_nx)$.
By Yoneda Lemma, we have an isomorphism $\sigma_n\cong\tau_n$.
\end{proof}

Note that Theorem $\ref{higher defect formula}$ is independently obtained by Jasso and Kvamme in \cite[Theorem 3.7, Corollary 3.8]{JK}.
The proof is different since we proved the higher defect formula in Theorem \ref{higher defect formula} without using the explicit form of $\tau_n$.

%%%%%%%%%%%%%%%%%%%%%%%%%%%%%%%%%%%%%%%%%%%%%%%%%%%%%%%%%%%%%%%%%%%%%%%%%

\end{document}